%% file: main.tex
\documentclass[12pt,reqno,a4paper]{amsart}
\usepackage[doi=true,isbn=false,style=alphabetic,sorting=nyt,backend=biber,maxnames=5,maxalphanames=5,giveninits=true]{biblatex}
\AtEveryBibitem{\clearlist{language}}
\bibliography{main}

\usepackage[breaklinks,colorlinks,plainpages,hypertexnames=false,plainpages=false]{hyperref}
\hypersetup{urlcolor=blue, citecolor=blue, linkcolor=blue}

\usepackage[utf8]{inputenc}
\usepackage{amsmath}
\usepackage{amsthm}
\usepackage{amssymb}
\usepackage{amsfonts}
\usepackage{enumerate}
\usepackage{mathtools}
\usepackage{tikz-cd}
\usepackage{hhline}
\usepackage{stmaryrd}
\usepackage{enumitem}
\usepackage{centernot}
\usepackage[textwidth=25mm]{todonotes}
\usepackage{mathrsfs}
\usepackage{comment}
\usepackage{listings}

\lstdefinelanguage{Julia}%
  {morekeywords={abstract,break,case,catch,const,continue,do,else,elseif,%
      end,export,false,for,function,immutable,import,importall,if,in,%
      macro,module,otherwise,quote,return,true,try,type,typealias,%
      using,while},%
   sensitive=true,%
   morecomment=[l]\#,%
   morecomment=[n]{\#=}{=\#},%
   morestring=[s]{"}{"},%
   morestring=[m]{'}{'},%
 }[keywords,comments,strings]%

\usepackage{multicol}
\usepackage[capitalise, noabbrev]{cleveref} 
\usepackage[noend]{algorithmic} 

\usepackage{mathdots}
\usepackage{cleveref} 
\usepackage{nicematrix} 
\usepackage{mathalpha} 
\DeclareMathAlphabet{\dutchcal}{U}{dutchcal}{m}{n}
\newcommand{\dcf}{\dutchcal{f}}
\newcommand{\dcF}{\mathcal{F}}

\newcommand{\dcG}{\mathcal{G}}
\newcommand{\dcI}{\mathcal{I}}

\usepackage{array}
\newcolumntype{L}[1]{>{\raggedright\let\newline\\\arraybackslash\hspace{0pt}}m{#1}}
\newcolumntype{C}[1]{>{\centering\let\newline\\\arraybackslash\hspace{0pt}}m{#1}}
\newcolumntype{R}[1]{>{\raggedleft\let\newline\\\arraybackslash\hspace{0pt}}m{#1}}

\usepackage{tikz}
\usetikzlibrary{arrows}
\usetikzlibrary{decorations.pathreplacing}
\usetikzlibrary{matrix}
\usetikzlibrary{calc}
\usetikzlibrary{shapes}
\usetikzlibrary{patterns}
\usetikzlibrary{fit,backgrounds,scopes}

\newtheoremstyle{theoremstyle}
{10pt}      
{5pt}       
{\itshape}  
{}          
{\bfseries} 
{}         
{ }      
{}          

\newtheoremstyle{algorithmstyle}
{10pt}      
{5pt}       
{}  
{}          
{\bfseries} 
{}         
{ }      
{}          

\newtheoremstyle{examplestyle}
{10pt}      
{5pt}       
{}          
{}          
{\bfseries} 
{}         
{ }      
{}          

\makeatletter
\newtheorem*{rep@theorem}{\rep@title}
\newcommand{\newreptheorem}[2]{%
\newenvironment{rep#1}[1]{%
 \def\rep@title{#2 \ref{##1}}%
 \begin{rep@theorem}}%
 {\end{rep@theorem}}}
\makeatother

\makeatletter 
\newcommand{\subalign}[1]{%
  \vcenter{%
    \Let@ \restore@math@cr \default@tag
    \baselineskip\fontdimen10 \scriptfont\tw@
    \advance\baselineskip\fontdimen12 \scriptfont\tw@
    \lineskip\thr@@\fontdimen8 \scriptfont\thr@@
    \lineskiplimit\lineskip
    \ialign{\hfil$\m@th\scriptstyle##$&$\m@th\scriptstyle{}##$\hfil\crcr
      #1\crcr
    }%
  }%
}
\makeatother

\newreptheorem{theorem}{Theorem}
\newreptheorem{corollary}{Corollary}
\newreptheorem{proposition}{Proposition}

\theoremstyle{theoremstyle}
\newtheorem{theorem}{Theorem}[section]
\newtheorem{lemma}[theorem]{Lemma}

\theoremstyle{examplestyle}
\newtheorem{example}[theorem]{Example}
\newtheorem{definition}[theorem]{Definition}
\newtheorem{experiment}[theorem]{Experiment}
\newtheorem{problem}[theorem]{Problem}
\newtheorem{remark}[theorem]{Remark}

\newtheorem{convention}[theorem]{Convention}

\theoremstyle{algorithmstyle}
\newtheorem{algorithm}[theorem]{Algorithm}

\newcommand{\CC}{\mathbb{C}}

\newcommand{\QQ}{\mathbb{Q}}
\newcommand{\ZZ}{\mathbb{Z}}

\newcommand{\suchthat}{\;\ifnum\currentgrouptype=16 \middle\fi|\;}

\newcommand{\bigslant}[2]{{\raisebox{.2em}{$#1$}\left/\raisebox{-.2em}{$#2$}\right.}}

\newcommand\floor[1]{\lfloor#1\rfloor}

\DeclareMathOperator{\codim}{codim}

\DeclareMathOperator{\Mac}{Mac}

\DeclareMathOperator{\sgn}{sgn}

\DeclareMathOperator{\Supp}{Supp}

\DeclareMathOperator*{\argmin}{arg\,min}

\begin{document}

\title[Embedding polynomial systems into vertically parametrised families]{Embedding polynomial systems into vertically parametrised families: A case study on ODEbase}
\author{Oliver Daisey}
\address{Department of Mathematics, Durham University, United Kingdom.}
\email{oliver.j.daisey@durham.ac.uk}
\urladdr{https://www.durham.ac.uk/staff/oliver-j-daisey/}
\author{Yue Ren}
\address{Department of Mathematics, Durham University, United Kingdom.}
\email{yue.ren2@durham.ac.uk}
\urladdr{https://yueren.de}
\author{Yuvraj Singh}
\address{Department of Mathematics, Durham University, United Kingdom.}
\email{yuvraj.singh@durham.ac.uk}
\urladdr{https://yuvraj-singh-math.github.io}

\maketitle

\begin{abstract}
  Vertically parametrised polynomial systems are a particular nice class of parametrised polynomial systems for which a lot of interesting algebraic information is encoded in its combinatorics.  Given a fixed polynomial system, we empirically study what constitutes a good vertically parametrised polynomial system that gives rise to it and how to construct said vertically parametrised polynomial system.  For data, we use all polynomial systems in \textsc{ODEbase}, which we have transcribed to an OSCAR readable format, and made available as a Julia package \textsc{OscarODEbase}.
\end{abstract}

\section{Introduction}
Vertically parametrised systems are an important class of parametrised polynomial systems.  They describe the steady states in mass-action kinetics, and many of their interesting properties are encoded in their (tropical) combinatorics:
\begin{enumerate}
\item the solution set always has the expected dimension and at least one smooth point \cite{FeliuHenrikssonPascual-Escudero2024},
\item if generically zero-dimensional, both
  \begin{enumerate}
  \item the generic number of complex solutions \cite{HelminckRen2022} and,
  \item a lower bound on the number of positive solutions for a particular instance \cite{RoseTelek2024}
  \end{enumerate}
  can be computed combinatorially using tropical geometry,
\item if generically zero-dimensional, optimal homotopies can be constructed combinatorially using tropical geometry \cite{HelminckHenrikssonRen2024}.
\end{enumerate}

In this article, given a polynomial system $F=\{f_1,\dots,f_k\}\subseteq K[x^\pm]$, we study the task of finding a ``good'' embedding of $F$ into a vertically parametrised system $\dcF$. Specifically, our goal is to find a vertically parametrised system $\dcF$ such that $F=\dcF_P$ for some choice of parameters $P\in K^m$ with $\dcF_P$ denoting the specialisation of $\dcF$ at $P$.  By ``good'', we mean that the system $F$ exhibits many generic properties of $\dcF$. For example, the number of solutions of $F$ should coincide with the generic root count of $\dcF$. This naturally leads to the first problem:

\begin{center}
  \textbf{Problem 1:}  Identify easy criteria for good embeddings.
\end{center}

In \cref{sec:Problem1}, we turn to the polynomial networks in \textsc{ODEbase} \cite{LudersSturmRadulescu2022}.  We consider a random specialisation of each vertically parametrised system and identify which features distinguish the original parametrised system from other embeddings of its specialisation.  We find that
\begin{enumerate}[label=(\alph*),leftmargin=*]
\item the original parametrised system minimises the number of distinct monomials,
\item there is no conclusive tiebreaker among all perturbations minimising the number of distinct monomials.
\end{enumerate}

This then naturally leads to the second problem
\begin{center}
  \textbf{Problem 2:} Construct a good embedding.
\end{center}

In \cref{sec:Problem2}, we show that constructing an optimal embedding is an incredibly challenging task.  Instead, we propose a greedy algorithm and compare its performance to the optimum.

In \cref{sec:ODEbase}, we introduce our OSCAR \cite{Oscar} interface to the \textsc{ODEbase} database and explain how it can be used.

\subsection*{Acknowledgments}
Yue Ren is supported by UK Research and Innovation Future Leaders Fellowship programme ``Computational Tropical Geometry and its Applications'' (MR/S034463/2).
\section{Background}\label{sec:background}
In this section, we go over some basic notions that are of immediate interest.

\subsection{Parametrised polynomial systems}
We begin with introducing some notation for parametrised polynomials.  Our notation will be the same as in \cite{HelminckHenrikssonRen2024}.

\begin{convention}
  \label{con:parametrisedPolynomials}
  For the entirety of the paper, we fix a \emph{parametrised (Laurent) polynomial ring} $K[a][x^\pm]\coloneqq K[a_1,\dots,a_m][x_1^\pm,\dots,x_n^\pm]$ with parameters $a_1,\dots,a_m$ and variables $x_1,\dots,x_n$.  Elements $\dcf\in K[a][x^\pm]$ are referred to as \emph{parametrised polynomials}, and ideals $\dcI\subseteq K[a][x^\pm]$ are referred to as \emph{parametrised polynomial ideals}.  Finite $\dcF=\{\dcf_1,\dots,\dcf_k\}\subseteq K[a][x^\pm]$ are \emph{parametrised polynomial systems}.

  Points $P=(P_1,\dots,P_m)\in K^m$ are called \emph{choices of parameters} and we use subscripts to denote specialisations thereover:
  \begin{align*}
    \dcf_P&\coloneqq \sum_{\alpha\in S} c_\alpha(P)\cdot x^\alpha \in K[x^\pm],\\
    \intertext{provided $\dcf=\sum_{\alpha\in S} c_\alpha\cdot x^\alpha$ with $c_\alpha\in K[a]$ and $S\subseteq \ZZ^n$ finite, and}
    \dcI_P&\coloneqq \langle \dcf_P\mid \dcf\in \dcI\rangle \subseteq K[x^\pm],\\
    \dcF_P&\coloneqq \{\dcf_{1,P},\dots,\dcf_{k,P}\} \subseteq K[x^\pm].
  \end{align*}
\end{convention}

\begin{definition}\
  \label{def:specialisation}
  The \emph{generic specialisation} of a parametrised polynomial $\dcf=\sum_{\alpha\in S}c_\alpha x^\alpha\in K[a][x^\pm]$ is the polynomial over the rational function field $K(a)\coloneqq K(a_1,\dots,a_m)$ given by
  \begin{equation*}
    \dcf_{K(a)} \coloneqq \sum_{\alpha\in S}c_\alpha x^\alpha \in K(a)[x^\pm].
  \end{equation*}
  The \emph{generic specialisation} of a parametrised polynomial ideal $\dcI\subseteq K[a][x^\pm]$ is the polynomial ideal over $K(a)$ given by
  \begin{equation*}
    \dcI_{K(a)}\coloneqq \langle \dcf_{K(a)}\mid \dcf \in \dcI\rangle\subseteq K(a)[x^\pm].
  \end{equation*}
\end{definition}

\begin{definition}
  \label{def:genericProperties}
  Let $\dcI\subseteq K[a][x^\pm]$ be a parametrised polynomial ideal.\\
  The \emph{generic dimension} of $\dcI$ is the Krull dimension of its generic specialisation $\dcI_{K(a)}$.  The \emph{generic root count} of $I$ is $\ell_{\dcI,K(a)}\coloneqq \infty$ if $\dim(\dcI_{K(a)})>0$ or it is the vector space dimension
  \begin{equation*}
    \ell_{\dcI,K(a)} \coloneqq \dim_{K(a)} \bigslant{K[x^\pm]}{\dcI_{K(a)}}.
  \end{equation*}
  The \emph{generic degree} of $\dcI$ is the generic root count of $\dcI+I_L$, where $I_L\subseteq K[x^\pm]$ is a generic linear ideal with $\codim(I_L)=\dim(\dcI_{K(a)})$.
\end{definition}

The generic root count can be regarded as the number of solutions of $\dcI_P$ for a generic choice of parameters $P\in K^m$ and counted with a suitable multiplicity \cite[Corollary 2.5]{CoxLittleOShea05}.  Consequently the generic degree counts the number of intersection points of $V(\dcI_P)\cap L$ for a generic choice of parameters $P\in K^m$ and with a generic linear space $L\subseteq (K^\ast)^n$ of complementary dimension.

\subsection{Vertically parametrised ideals}
Next we introduce vertically parametrised ideals and embeddings of polynomial systems into them.  We recall a central notion of \cite{FeliuHenrikssonPascual-Escudero2024}, which implies that the (generic) dimension is independent of the embedding.

\begin{definition}
  \label{def:verticallyParametrisedSystems}\
  A \emph{vertically parametrised} polynomial system is a system $\dcF=\{\dcf_1,\dots,\dcf_k\}\subseteq K[a][x^\pm]$ of the form
  \begin{equation}
    \label{eq:parametrisedSystem}
    \dcf_i\coloneqq \sum_{j=1}^m c_{i,j}a_j x^{\alpha_j}
  \end{equation}
  for some finite set $S=\{\alpha_1,\dots,\alpha_m\}\subseteq\ZZ^n$ and some $c_{i,j}\in K$.  We generally assume the $\alpha_j$ to be pairwise distinct, $k\leq n\leq m$, and that for all $\alpha_j\in S$ there is an $i\in[k]$ with $c_{i,j}\neq 0$.
  We will refer to the ideal $\dcI\coloneqq \langle\dcF\rangle\subseteq K[a][x^\pm]$ as a \emph{vertically parametrised} ideal.
\end{definition}

\begin{definition}
  \label{def:embedding}
  Let $F=\{f_1,\dots,f_k\}\subseteq K[x^\pm]$ be a polynomial system, say $f_i=\sum_{j=1}^m c_{i,j} x^{\alpha_j}$ for some $S=\{\alpha_1,\dots,\alpha_m\}\subseteq \ZZ^n$, and let $I\coloneqq\langle F\rangle\subseteq K[x^\pm]$ be the ideal it generates.
  The \emph{minimal vertical system} specialising to $F$ is the parametrised system $\dcF$ from \cref{eq:parametrisedSystem}.  We denote its parametrised polynomial ideal by $\dcI_F\coloneqq\langle \dcF\rangle$, it naturally specialises to $I$.  We will refer to $\dcI$ as a \emph{vertical family} containing $I$, as it can be regarded as a set of polynomial ideals of which one is $I$.
\end{definition}

\begin{example}
  \label{ex:embedding}
  Consider the polynomial ideal
  \begin{equation*}
    I\coloneqq \langle \underbrace{x_1^2+x_2^2+x_1 , x_1^2+x_2^2+1}_{\eqqcolon F} \rangle = \langle \underbrace{x_1^2+x_2^2+x_1 , x_1^2x_2+x_2^3+x_2}_{\eqqcolon G} \rangle\subseteq \CC[x_1^\pm, x_2^\pm],
  \end{equation*}
  and let $F,G\subseteq \CC[x^\pm]$ be the two polynomial systems above.  The minimal parametrised systems specialising to $F$ and $G$ are
  \begin{equation*}
    \dcF\coloneqq
    \left\{
      \begin{array}{l}
        a_1x_1^2+a_2x_2^2+a_3x_1\\
        a_1x_1^2+a_2x_2^2+a_4
      \end{array}
    \right\}
    \qquad \text{and} \qquad
    \dcG\coloneqq
    \left\{
      \begin{array}{l}
        b_1x_1^2 + b_2x_2^2+b_3x_1\\
        b_4x_1^2x_2+b_5x_2^3+b_6x_2
      \end{array}
    \right\}.
  \end{equation*}
  They give rise to two distinct vertically parametrised ideals $\dcI_{F}\coloneqq\langle\dcF\rangle\subseteq K[a][x^\pm]$ and $\dcI_{G}\coloneqq\langle\dcG\rangle\subseteq K[b][x^\pm]$ in distinct parametrised polynomial rings $K[a][x^\pm]\coloneqq K[a_1,\dots,a_4][x^\pm]$ and $K[b][x^\pm]\coloneqq K[b_1,\dots,b_6][x^\pm]$.  It is straightforward to see that
  \begin{equation*}
    \ell_I = \ell_{\dcI_F,K(a)} = 2 < 4 = \ell_{\dcI_G,K(b)}.
  \end{equation*}
  The reason why the generic root count of $\dcI_F$ is lower than that of $\dcI_G$ can be attributed to cancellations in the higher degree monomials of $\dcF$.
\end{example}

We close this subsection with two results from existing literature on minimal vertical systems. \cref{lem:dimension} shows that the generic dimension of the minimal vertical system is independent on the choice of system.  \cref{lem:degree} shows that the generic degree of the minimal vertical system is increasing in the monomial support.

\begin{lemma}
  \label{lem:dimension}
  Let $I\subseteq K[x^\pm]$ be a complete intersection of codimension $r$, and let $F=\{f_1,\dots,f_r\}$ be any generating set of $I$.  Let $\dcI_F\subseteq K[a][x^\pm]$ be the vertically parametrised ideal induced by $F$, and let $\dcI_{F,K(a)}$ denote the generic specialisation.  Then $\dim(\dcI_{F,K(a)})=n-r$.
\end{lemma}
\begin{proof}
  By \cite[Theorem 3.1]{feliu2024genericconsistencynondegeneracyvertically}, the incidence variety
  \begin{equation*}
    \Big\{(P,x) \in K^m \times (K^{*})^n \mid x \in V(\dcI_{F,P})\Big\}
  \end{equation*}
  is irreducible of dimension $m+n-r$. Thus by \cite[Theorem 2.2]{feliu2024genericconsistencynondegeneracyvertically}, the generic specialisation $\dcI_{F, K(a)}$ has dimension $n-r$.
\end{proof}

\begin{lemma}
  \label{lem:degree}
  Let $I\subseteq K[x^\pm]$ be a complete intersection of codimension $r$, and let $F=\{f_1,\dots,f_r\}, G=\{g_1,\dots,g_r\}$ be two generating sets of $I$.  Let $\dcI_F, \dcI_G\subseteq K[a][x^\pm]$ be the vertically parametrised ideals induced by $F$ and $G$, respectively.  Suppose we have $\Supp(f_i)\subseteq\Supp(g_i)$ for all $i\in[r]$.  Then $\ell_{\dcI_F,K(a)}\leq \ell_{\dcI_G,K(a)}$.
\end{lemma}
\begin{proof}
  Follows from \cite[Lemma 5.2]{HelminckRen2022}.
\end{proof}
\subsection{Macaulay matrices}
One way to formalise the cancellations in \cref{ex:embedding} is using Macaulay matrices and their minors.

\begin{definition}
  \label{def:macaulayMatrix}
  Let $F=\{f_1,\dots,f_k\}\subseteq K[x^\pm]$ be a polynomial system, say $f_i=\sum_{j=1}^m c_{i,j} x^{\alpha_j}$ for some $S=\{\alpha_1,\dots,\alpha_m\}\subseteq \ZZ^n$.
  The \emph{Macaulay matrix} of $F$ is the matrix $\Mac(F)\coloneqq (c_{i,j})_{i\in[k],j\in[m]}\in K^{k\times m}$. A minor of the Macaulay-matrix $\det((c_{i,j})_{i\in[k],j\in J})$, $J\in\binom{[m]}{k}$, is \emph{non-trivial} if we can write $J\eqqcolon\{j_1,\dots,j_k\}$ with $c_{i,j_i}\neq 0$, and \emph{trivial} otherwise.
\end{definition}

The intuition behind a non-trivial minor of a Macaulay matrix is that every polynomial contributes to it.  Indeed, one can easily show that trivial minors have to vanish:
\begin{lemma}
  \label{lem:trivialMinors}
  Let $\det((c_{i,j})_{i\in[k],j\in J})$ be a trivial minor of the Macaulay matrix $\Mac(F)$.  Then $\det((c_{i,j})_{i\in[k],j\in J})=0$.
\end{lemma}
\begin{proof}
  Suppose $J=\{j_1,\dots,j_k\}$. By the Leibniz formula for determinants, if
  \begin{equation*}
    \det((c_{i,j})_{i\in[k],j\in J}) = \sum_{\sigma\in S_k}\sgn\left(\sigma\right)\prod_{i=1}^{k}c_{i,j_{\sigma(i)}} \neq 0,
  \end{equation*}
  then at least one of the summands has to be non-zero, showing that the minor was non-trivial.
\end{proof}

\section{Finding criteria for good embeddings}\label{sec:Problem1}
In this section, given a polynomial ideal $I\subseteq K[x^\pm]$, we discuss what distinguishes a good vertical family $\dcI$ containing $I$ from a bad vertical family containing $I$.  Ideally, the generic specialisation  $\dcI_{K(a)}$ should have a lot of properties in common with $I$.

Two of the most fundamental properties of polynomial ideals (or their affine varieties) are their dimension and degree. By \cref{lem:dimension}, the dimension is the same for all vertical families, whereas, by \cref{ex:embedding}, the degree is not.

\subsection{Experiment setup}
Our discussion will be based on experiments in which we will focus on the following features and scores of Macaulay matrices:

\begin{definition}
  \label{def:featuresAndScores}
  Let $F\subseteq K[x^\pm]$ be a polynomial system.  We denote
  \setlist[description]{font=\normalfont\hspace{5mm}}
  \begin{description}[itemsep=1mm]
  \item[$M(F)$] the number of minors of $\Mac(F)$ ($=\binom{m}{k}$).
  \item[$M_{0}(F)$] the number of zero minors of $\Mac(F)$.
  \item[$M^{\mathrm{nt}}(F)$] the number of non-trivial minors of $\Mac(F)$.
  \item[$M_{0}^{\mathrm{nt}}(F)$] the number of non-trivial zero minors of $\Mac(F)$.
  \end{description}
  and consider the following \emph{scores}:
  \begin{equation*}
    \setlength{\arraycolsep}{2pt}
    \begin{array}{llllllll}
      S(F)&\coloneqq -M(F), & \quad\quad\quad & S_0(F)&\coloneqq M_0(F), & \quad\quad\quad & R_0(F)&\coloneqq \displaystyle\frac{M_0(F)}{M(F)}, \\
          & & & S_0^{\mathrm{nt}}(F) &\coloneqq M_0^{\mathrm{nt}}(F), & & R_0^{\mathrm{nt}}(F) &\coloneqq \displaystyle\frac{M_0^{\mathrm{nt}}(F)}{M^{\mathrm{nt}}(F)}.
    \end{array}
  \end{equation*}
  For $R_0(F)$ and $R_0^{\mathrm{nt}}(F)$, note that we always have $M(F)\neq 0\neq M^{\mathrm{nt}}(F)$.
\end{definition}

\begin{example}
  \label{ex:embeddingScores}
  Consider $F$ and $G$ from \cref{ex:embedding}.  Their Macaulay matrices are
  \begin{equation*}
    \Mac(F)=
    \begin{pmatrix}
      1 & 1 & 1 & 0 \\
      1 & 1 & 0 & 1
    \end{pmatrix}
    \qquad \text{and} \qquad
    \Mac(G)=
    \begin{pmatrix}
      1 & 1 & 1 & 0 & 0 & 0 \\
      0 & 0 & 0 & 1 & 1 & 1
    \end{pmatrix}.
  \end{equation*}
  The scores of $F$ and $G$ are:
  \begin{equation*}
    \begin{array}{c|ccccc}
        & S & S_0 & S_0^{\mathrm{nt}} & R_0 & R_0^{\mathrm{nt}} \\[2mm] \hline
      F & -\binom{4}{2} & 1 & 1 & \frac{1}{6} & 1 \\[2mm]
      G & -\binom{6}{2} & 6 & 0 & \frac{6}{15} & 0
    \end{array}
  \end{equation*}
\end{example}

\begin{remark}
  \label{rem:scores}
  The intuition behind the scores in \cref{def:featuresAndScores} is as follows:

  Maximizing $S(\cdot)$ minimizes the monomial support of the system.  This can be beneficial for minimizing the degree, as seen in \cref{lem:degree}.

  Maximizing the remaining scores maximizes vanishing of minors in some sense. This is motivated by \cref{ex:embeddingScores}, as well as the results in \cite[Section 6.1]{HelminckRen2022}, which show that the generic root count of a vertically parametrised ideal can be expressed as a tropical intersection product of a tropical linear space and a tropicalized binomial variety.  Said intersection product depends on how ``simple'' the tropical linear space is, see \cref{fig:tropicalIntersectionDifferentLinearSpaces}, which is related to how many entries of its Pl\"ucker vector are zero.
\end{remark}

\begin{figure}[t]
  \centering
  \begin{tikzpicture}
    \node (left) at (-4,0)
    {
      \begin{tikzpicture}[scale=0.7]
        \draw[very thick,blue!50!black] (-2,-1) -- (2,1);
        \draw[->,very thick,blue!50!black] (0,0) -- (1,0.5);
        \node[anchor=south east,font=\footnotesize,blue!50!black] at (1,0.5) {$(2,1)$};

        \draw[very thick,red!50!black] (-2,0) -- (2,0);
        \draw[->,very thick,red!50!black] (0,0) -- (1,0);
        \node[anchor=north,font=\footnotesize,red!50!black,yshift=-1mm] at (1,0) {$(0,1)$};
        \fill[white,draw=black] (0,0) circle (3pt);
        \node[anchor=north,yshift=-1mm] at (0,0) {1};
      \end{tikzpicture}
    };

    \node (mid) at (0,0)
    {
      \begin{tikzpicture}[scale=0.7]
        \draw[very thick,blue!50!black] (-2,-1) -- (2,1);
        \draw[->,very thick,blue!50!black] (0,0) -- (1,0.5);
        \node[anchor=north west,font=\footnotesize,blue!50!black] at (1,0.5) {$(2,1)$};

        \draw[very thick,red!50!black] (0,-2) -- (0,2);
        \draw[->,very thick,red!50!black] (0,0) -- (0,1);
        \node[anchor=east,font=\footnotesize,red!50!black,xshift=-1mm] at (0,1) {$(0,1)$};
        \fill[white,draw=black] (0,0) circle (3pt);
        \node[anchor=north west] at (0,0) {2};
      \end{tikzpicture}
    };

    \node (right) at (4,0)
    {
      \begin{tikzpicture}[scale=0.7]
        \draw[very thick,blue!50!black] (-2,-1) -- (2,1);

        \draw[very thick,red!50!black]
        (-2,-0.5) -- (0,-0.5)
        (0,-0.5) -- (0,-2)
        (0,-0.5) -- (2,1.5);
        \fill[white,draw=black] (-1,-0.5) circle (3pt);
        \node[anchor=north] at (-1,-0.5) {1};
        \fill[white,draw=black] (1,0.5) circle (3pt);
        \fill[white,draw=black] (1,0.5) circle (3pt);
        \node[anchor=north] at (1,0.5) {1};
      \end{tikzpicture}
    };
  \end{tikzpicture}
  \caption{The tropical intersection products of one tropicalized binomial ideal (blue) and its intersection product with three different tropical lines (red).}
  \label{fig:tropicalIntersectionDifferentLinearSpaces}
\end{figure}
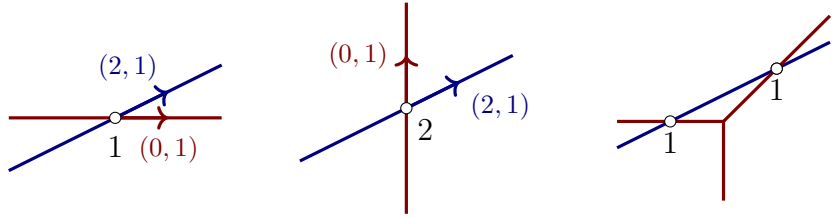

In the experiments, we will be comparing generic specialisations of parametrised polynomial systems to their perturbations in the following sense:

\begin{definition}
  \label{def:perturbation}
  Suppose $F=\{f_1,\dots,f_k\}$.  A \emph{perturbation} of $F$ is a system of the form
  \begin{equation}
    \label{eq:perturbation}
    F'\coloneqq \big\{ f_1,\dots,f_{i-1},x_j^s \cdot f_i,f_{i+1},\dots,f_n \big\} \quad \text{for some } j\in[k], i\in[n], s\in\{0,\pm1\}.
  \end{equation}
\end{definition}

\subsection{Experimental results}
\label{exp:features}
The goal of this experiment is to identify which of the scores described in \cref{def:featuresAndScores} is best at distinguishing the systems of \textsc{ODEbase} from their non-trivial perturbations.

There are $51$ systems arising from mass-action kinetics in \textsc{ODEbase} that have toric solutions (i.e. \(f_i\) is non-monomial for all \(f_i\in F\)) of which $31$ have $16$ species or less.  These are the systems we consider.

For each system $F$, we test whether any perturbation $F'$ as in \cref{eq:perturbation} has a higher score than the original system, i.e., we test how successful each score is at identifying the original system amongst its perturbations.  The result is:
\begin{equation*}
  S\colon 28 \qquad S_0\colon 2 \qquad S_0^\mathrm{nt}\colon 9 \qquad R_0\colon 9 \qquad R_0^\mathrm{nt}\colon 2
\end{equation*}
We see that score $S$ performs best out of all scores in \cref{def:featuresAndScores}, i.e., that the original system tends to have the fewest number of unique monomials.

Note however that original system is not always the best with respect to $S$:
According to \textsc{ODEbase}, the system \texttt{BIOMD0000000629} has the following reaction network ($s_1,\dots,s_5$ denoting the species and $\kappa_2,\dots,\kappa_5$ denoting the reaction rates):
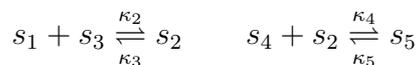
\begin{equation*}
  s_1+s_3 \underset{\kappa_3}{\overset{\kappa_2}{\rightleftharpoons}} s_2 \qquad s_4+s_2\underset{\kappa_5}{\overset{\kappa_4}{\rightleftharpoons}} s_5
\end{equation*}
leading to the following ODE polynomials:
\begin{align*}
  f_1\coloneqq \dot x_1 &= -\kappa_2 x_1x_3 + \kappa_3x_2, &
  f_4\coloneqq \dot x_4 &= -\kappa_4 x_2x_4 + \kappa_5x_5,\\
  f_2\coloneqq \dot x_2 &= \kappa_2 x_1x_3 - \kappa_4x_2x_4 - \kappa_3x_2 +\kappa_5x_5, &
  f_5\coloneqq \dot x_5 &= \kappa_4 x_2x_4 - \kappa_5x_5. \\
  f_3\coloneqq \dot x_3 &= -\kappa_2 x_1x_3 + \kappa_3x_2,
\end{align*}
and the following three constraints:
\begin{align*}
  \kappa_1x_1 + \kappa_1x_2 + \kappa_1x_5 - \kappa_1\kappa_6&=0, &
  \kappa_1x_2 + \kappa_1x_3 + \kappa_1x_5 - \kappa_1\kappa_7&=0, \\
  \kappa_1x_4 + \kappa_1x_5 - \kappa_1\kappa_8&=0.
\end{align*}
According to the constraints under consideration, $x_1,x_2,x_4$ are uniquely determined by $x_3,x_5$, hence only $f_3,f_5$ are necessary. But evidently the system \(\left\{ f_3,f_5 \right\}\) has more unique monomials than \(\left\{ x_4\cdot f_3,f_5 \right\}\), showing that the original system is not always the best with respect to $S$.

Moreover, observe that in the cases where the original system attains the best score in $S$, it need not be the unique maximum among its perturbations:  After taking the constraints into consideration, the system \texttt{BIOMD0000000405} has the following necessary ODE polynomials:
\begin{align*}
  f_2\coloneqq \dot x_2&=-\kappa_4x_2x_5 - \kappa_5x_2 + \kappa_2, &
  f_3\coloneqq \dot x_3&= \kappa_4x_1x_5 - \kappa_3x_3, \\
  f_4\coloneqq \dot x_4&=\kappa_4x_2x_5 - \kappa_3x_4, &
  f_5\coloneqq \dot x_5&=-\kappa_4x_1x_5 - \kappa_4x_2x_5 + \kappa_3x_3 + \kappa_3x_4.
\end{align*}
One can observe that multiplying \(f_2\) by \(x_3\), \(x_4\) or \(x_5\) preserves the number of unique monomials, hence the score \(S\) is the same for these perturbations.

Experiments aimed at identifying a score that can act as a tiebreaker for $S$ were inconclusive. There are 9 models which are not strict maximums. Out of these, 5 models, including \texttt{BIOMD0000000405}, have the original system lose on every other score to another perturbation. Thus there is no clear way of identifying the original system out of other systems with the same score.

\section{Finding good embeddings}\label{sec:Problem2}

In \cref{exp:features}, we identified that one of the primary features of a good embedding is that it minimizes the number of distinct monomials.  This means that finding an optimal embedding entails finding an optimal solution to the following problem:

\begin{problem}
  \label{prob:alignmentProblem}
  Given finite $S_1,\dots,S_k\subseteq\ZZ^n$, find translates $S_1+v_1,\dots,S_k+v_k\subseteq\ZZ^n$ that minimize $|\bigcup_{i=1}^k (S_i+ v_i)|$.
\end{problem}

Unfortunately, \cref{prob:alignmentProblem} is a well-known challenging problem without an efficient solution:

\begin{remark}
  \label{rem:alignmentProblem}
  The problem of optimally aligning two convex polytopes in 2D and 3D with $n$ total vertices by translations has been studied by de Berg et al. \cite{de1996computing} and Ahn et al. \cite{ahn2008maximum}, providing algorithms that run in expected $O(n\log n)$ and $O(n^3\log^4 n)$ time respectively. Allowing the ambient dimension $d$ to vary, Ahn et al. provide a probabilistic algorithm that finds a maximal overlap in $O(m^{\floor{d/2}+1}\log ^d m)$, where $m$ is the number of defining hyperplanes for the two polytopes \cite{ahn2013maximum}. In the 3D case, this specialises to $O(m^3 \log^{3.5} m)$.

  The difficulty of our problem becomes clear when we allow both the number of dimensions $d$ and number of polytopes to grow, and here it is difficult to provide complexity estimates. Computation of $d$-volume is already \#P-hard, and to our knowledge the best algorithm for our problem in the literature is an oracle polynomial time ellipsoid method \cite{fukuda2007polynomial}. One can hope for better bounds if one only asks for a $(1-\varepsilon)$-approximation of the optimal overlap for some $\varepsilon>0$ sufficiently small, in which case Ahn et al. \cite{ahn2005maximizing} offer an algorithm for two convex polytopes in the plane that runs in $O(1/\varepsilon \log(n/\varepsilon))$ time.
\end{remark}

Instead of solving \cref{prob:alignmentProblem}, we therefore propose the following greedy algorithm:

\begin{algorithm}[\texttt{GreedyAlignment}]\label{alg:greedyAlignment}\
  \begin{algorithmic}[1]
    \REQUIRE{$(S_1,\dots,S_k)$, where $S_i\subseteq \ZZ^n$ finite.}
    \ENSURE{$(S_1,S_2+v_2,\dots,S_k+v_k)$ such that each $v_\ell\in\ZZ^n$ minimises $|\bigcup_{i=1}^\ell (S_{i}+v_{i})|$ with $v_1\coloneq 0$.}
    \STATE Compute
    \begin{equation*}
      v_2\coloneqq \argmin_{v\in S_1-S_2}\Big(|S_1\cup (S_2+v)|\Big)
    \end{equation*}
    where $S_1-S_2\coloneqq \{ \alpha_1-\alpha_2\mid \alpha_1\in S_1, \alpha_2\in S_2\}$.
    \RETURN{$(0,\alpha_2,\texttt{GreedyAlignment}(S_1\cup(S_2+v_2),S_3,\dots,S_k))$}
  \end{algorithmic}
\end{algorithm}

Note that \cref{alg:greedyAlignment} is not guaranteed to find the optimal solution:

\begin{example}
  \label{ex:greedyNotOptimal}
  For the sake of simplicity, we will consider the overdetermined case of three polynomials in two variables.  Consider the system of polynomials
  \begin{equation*}
    f_1\coloneqq 1+x+xy,\quad f_2\coloneqq y+x+xy,\quad f_3\coloneqq x+y+x^2y+y^2x,
  \end{equation*}
  with a score given by counting the number of overlapping points. \cref{fig:greedyNotOptimal} shows the supports of the polynomial above as point sets in $\ZZ^2$, as well as their optimal alignment and a non-optimal alignment obtained by greedily aligning $S_1$ and $S_2$ first.
\end{example}

\begin{figure}
  \centering
  \begin{tikzpicture}
    \node[anchor=south east,xshift=-2mm,yshift=1mm] (set1) at (0,0)
    {
      \begin{tikzpicture}[scale=0.7]
        \draw (-0.5,-0.5) rectangle (1.5,1.5);
        \fill
        (0,0) circle (3pt)
        (0,1) circle (3pt)
        (1,1) circle (3pt);
      \end{tikzpicture}
    };
    \node[anchor=south] at (set1.north) {$S_1$};
    \node[anchor=south west,xshift=2mm,yshift=1mm] (set2) at (0,0)
    {
      \begin{tikzpicture}[scale=0.7]
        \draw (-0.5,-0.5) rectangle (1.5,1.5);
        \fill
        (0,0) circle (3pt)
        (1,0) circle (3pt)
        (1,1) circle (3pt);
      \end{tikzpicture}
    };
    \node[anchor=south] at (set2.north) {$S_2$};
    \node[anchor=north,yshift=-1mm] (set3) at (0,0)
    {
      \begin{tikzpicture}[scale=0.7]
        \draw (-0.5,-0.5) rectangle (2.5,2.5);
        \fill
        (1,0) circle (3pt)
        (0,1) circle (3pt)
        (1,2) circle (3pt)
        (2,1) circle (3pt);
      \end{tikzpicture}
    };
    \node[anchor=north] at (set3.south) {$S_3$};
    \node[anchor=south,yshift=2mm] (alignment1) at (5,0)
    {
      \begin{tikzpicture}[scale=1]
        \draw (-0.5,-0.5) rectangle (2.5,2.5);
        \fill
        (1,0) circle (2pt)
        (0,1) circle (2pt)
        (1,1) circle (2pt)
        (1,2) circle (2pt)
        (2,1) circle (2pt);
        \node[anchor=north,font=\scriptsize] at (1,0) {$1,3$};
        \node[anchor=north,font=\scriptsize] at (0,1) {$2,3$};
        \node[anchor=north,font=\scriptsize] at (1,1) {$1,2$};
        \node[anchor=north,font=\scriptsize] at (2,1) {$1,3$};
        \node[anchor=north,font=\scriptsize] at (1,2) {$2,3$};
      \end{tikzpicture}
    };
    \node[anchor=south] at (alignment1.north) { optimal alignment };
    \node[anchor=north,yshift=-2mm] (alignment2) at (5,0)
    {
      \begin{tikzpicture}[scale=1]
        \draw (-0.5,-0.5) rectangle (2.5,2.5);
        \fill
        (0,0) circle (2pt)
        (1,0) circle (2pt)
        (0,1) circle (2pt)
        (1,1) circle (2pt)
        (1,2) circle (2pt)
        (2,1) circle (2pt);
        \node[anchor=north,font=\scriptsize] at (0,0) {$2$};
        \node[anchor=north,font=\scriptsize] at (1,0) {$1,2,3$};
        \node[anchor=north,font=\scriptsize] at (0,1) {$3$};
        \node[anchor=north,font=\scriptsize] at (1,1) {$1,2$};
        \node[anchor=north,font=\scriptsize] at (2,1) {$1,3$};
        \node[anchor=north,font=\scriptsize] at (1,2) {$3$};
      \end{tikzpicture}
    };
    \node[anchor=north] at (alignment2.south) { greedy alignment };
  \end{tikzpicture}
  \caption{The optimal alignment of $S_1, S_2, S_3$, and an unoptimal alignment obtained by greedily aligning $S_1$ and $S_2$ first.}
  \label{fig:greedyNotOptimal}
\end{figure}
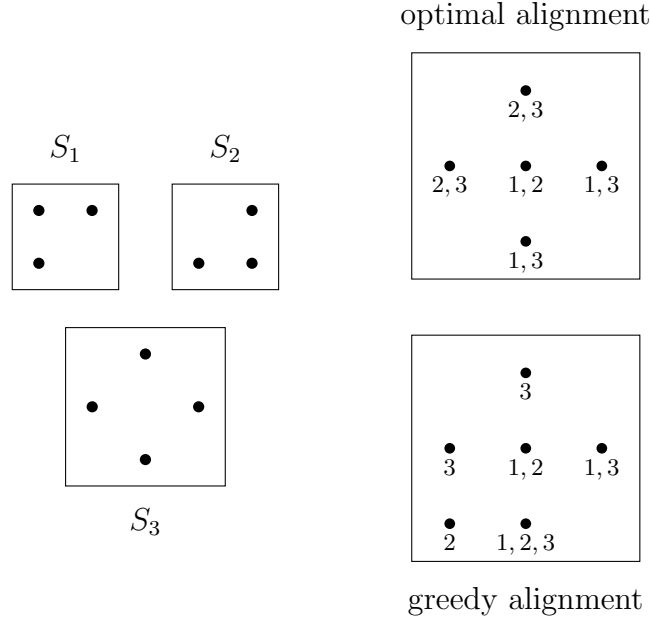

However, \cref{alg:greedyAlignment} performs surprisingly well on systems in practice:

\begin{experiment}
  \label{ex:greedyAlignment}
  We consider all 70 systems in \textsc{ODEbase} arising from mass-action kinetics that do not have parameters in the exponent.  For each system $F=\{f_1,\dots,f_k\}$, we run \cref{alg:greedyAlignment} on the monomial supports of $10$ random translations
  \begin{equation*}
    x^{\alpha_1}\cdot f_1,\quad \cdots, \quad x^{\alpha_k}\cdot f_k \qquad \text{for } \alpha_i\in [0,2^8-1]^n \text{ random},
  \end{equation*}
  and compare how the greedily aligned system compares to the original system in terms of the number of monomials.  The data can be found in
  \begin{center}
    \url{https://github.com/yuvraj-singh-math/monomial-translation}
  \end{center}
  and is illustrated in \cref{fig:greedyAlignmentExperiment}.  We see that in $91\%$ of the cases, the average score is within $1.149$ of the optimal whereas the best-of-ten score is within $1.059$ of the optimal score. We therefore recommend running \cref{alg:greedyAlignment} multiple times on different random translations.

  We note that \texttt{BIOMD0000000205}, the 70th system in our filtered list, proved to be too memory intensive in the current program. We therefore treated this system differently in our experiments.
\end{experiment}

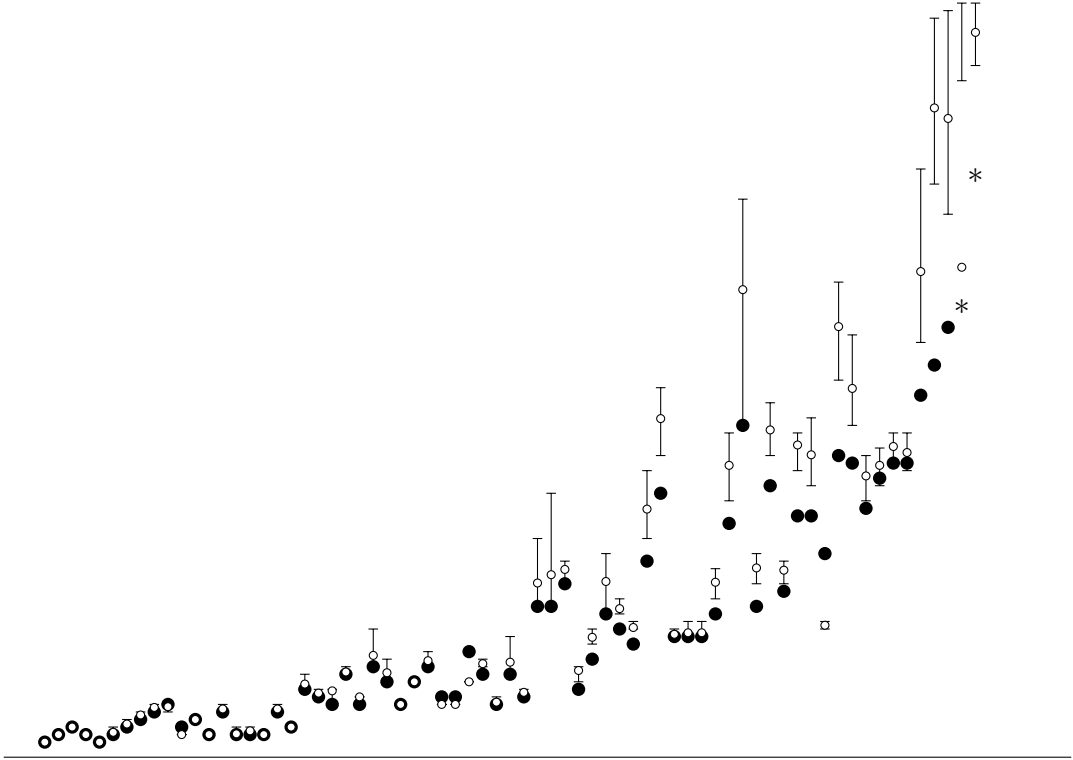
\begin{figure}[t]
  \centering
  \input{graph}
\caption{Results of \cref{alg:greedyAlignment} in \cref{ex:greedyAlignment}. Black points are the original and optimal score, white points are average of 10 runs, and the bars indicate the spread of scores. Black points replaced by asteriks indicate systems whose scores have been scaled by a constant factor to fit the figure.}\label{fig:greedyAlignmentExperiment}
\end{figure}

\section{OscarODEbase.jl}\label{sec:ODEbase}
In order to facilitate our experiments, we have created a small Julia package \textsc{OscarODEbase.jl}, which as of the date of writing contains 190 out of 200 polynomial models in \textsc{ODEbase} as OSCAR polynomials.  The missing polynomial models have features that are currently not supported in OSCAR, such as parameters in the exponents.

The package can be installed via git (remove line breaks):

\begin{lstlisting}[language=julia]
julia> Pkg.add(url="https://github.com/yuvraj-singh-math/
                    OscarODEbase.git")
\end{lstlisting}

Upon loading \texttt{OscarODEbase.jl}, the global variable \texttt{ODEbaseModels} contains a list of model names, which can be loaded using \texttt{get\_odebase\_model}:

\begin{lstlisting}[language=julia]
julia> using OscarODEbase;
julia> ODEbaseModels
190-element Vector{String}:
 "BIOMD0000000002"
 ...
 "BIOMD0000001038"
julia> model=get_odebase_model("BIOMD0000000854")
Entry BIOMD0000000854, with 4 species and 11 parameters.
Gray2016 - The Akt switch model
\end{lstlisting}

Each model is a struct of type \texttt{ODEbaseModel}, containing various information from \textsc{ODEbase}.  The following functions simply return the same-named information as stored in \textsc{ODEbase}:

\begin{center}
  \texttt{get\_ID},
  \texttt{get\_description},
  \texttt{get\_deficiency},
  \texttt{get\_stoichiometric\_matrix},
  \texttt{get\_reconfigured\_stoichiometric\_matrix},
  \texttt{get\_kinetic\_matrix}
\end{center}

\noindent For example:
\begin{lstlisting}[language=julia]
julia> get_stoichiometric_matrix(model)
[ 1   -1    1    0    0    0    0]
[-1    0    0    1   -1    0    0]
[ 0    1   -1    0    0    1   -1]
[ 0    0    0   -1    1   -1    1]
\end{lstlisting}

The next functions return the same-named information as stored in \textsc{ODEbase}, but as \textsc{OSCAR} polynomial data:

\begin{description}
  \item[\texttt{get\_parameter\_ring}] Return a parameter ring $\QQ[\kappa_1,\dots,\kappa_m]$ where the $\kappa_j$ represent reaction rates.
  \item[\texttt{get\_polynomial\_ring}] Return the parametrised polynomial ring $\QQ[\kappa][x]\coloneqq$ \linebreak $\QQ[\kappa_1,\dots,\kappa_m][x_1,\dots,x_n]$ where the $\kappa_j$ and $x_i$ represent reaction rates and species concentrations, respectively.
  \item[\texttt{get\_ODEs}] Return the steady state polynomials in $\QQ[\kappa][x]$, one per species.
  \item[\texttt{get\_constraints}] Return the constraints in $\QQ[\kappa][x]$.
\end{description}

\noindent For example:

\begin{lstlisting}[language=julia]
julia> get_ODEs(model)
4-element Vector{AbstractAlgebra.Generic.MPoly{...}}:
 -k1*k3*x1 + k2*x2 + k1*x3
 (-k1*k3 - k2)*x2 + k1*x4
 k1*k3*x1 + (-k1 - k2*k5)*x3 + k2*x4
 k1*k3*x2 + k2*k5*x3 + (-k1 - k2)*x4
\end{lstlisting}

Finally, the following two functions
\begin{center}
\texttt{get\_polynomials\_random\_specialization}

\texttt{get\_polynomials\_fixed\_specialization}
\end{center}
return a specialisation by choosing parameters uniformly in the range of \texttt{Int8}, and choosing parameters according to their values in \texttt{ODEbase} respectively.

Moreover, both functions have two optional Boolean parameters \texttt{constraint} and \texttt{reduce} which are \texttt{false} by default.  If \texttt{constraint==true}, both functions will add specialisations of the constraints polynomials.  If \texttt{reduce==true}, both functions will omit certain ODE polynomials that become redundant when adding the constraints. To be precise, it will omit all ODE polynomials whose species is a pivot of a row-echelon form of the coefficient matrix of the constraints polynomials.  Setting both \texttt{constraint} and \texttt{reduce} to \texttt{true} yields a square polynomial system.

The package and up-to-date instructions can be found at
\begin{center}
  \href{https://github.com/yuvraj-singh-math/OscarODEbase}{https://github.com/yuvraj-singh-math/OscarODEbase}.
\end{center}

\printbibliography
\end{document}

%% file: graph.tex
\begin{tikzpicture}[x={(1.8mm,0mm)}, y={(0mm,1mm)}] 
    \draw (-3,0) -- (75,0);

    \coordinate (original0486) at (0,2);
    \coordinate (best0486) at (0,2);
    \coordinate (average0486) at (0,2.0);
    \coordinate (worst0486) at (0,2);

    \draw
    (best0486)++(-0.35,0) -- ++(0.7,0)  
    (worst0486)++(-0.35,0) -- ++(0.7,0) 
    (best0486) -- (worst0486);       
    \fill[black] (original0486) circle (2.5pt);
    \fill[white,draw=black] (average0486) circle (1.5pt);

    \coordinate (original0092) at (1,3);
    \coordinate (best0092) at (1,3);
    \coordinate (average0092) at (1,3.0);
    \coordinate (worst0092) at (1,3);

    \draw
    (best0092)++(-0.35,0) -- ++(0.7,0)  
    (worst0092)++(-0.35,0) -- ++(0.7,0) 
    (best0092) -- (worst0092);       
    \fill[black] (original0092) circle (2.5pt);
    \fill[white,draw=black] (average0092) circle (1.5pt);

    \coordinate (original0233) at (2,4);
    \coordinate (best0233) at (2,4);
    \coordinate (average0233) at (2,4.0);
    \coordinate (worst0233) at (2,4);

    \draw
    (best0233)++(-0.35,0) -- ++(0.7,0)  
    (worst0233)++(-0.35,0) -- ++(0.7,0) 
    (best0233) -- (worst0233);       
    \fill[black] (original0233) circle (2.5pt);
    \fill[white,draw=black] (average0233) circle (1.5pt);

    \coordinate (original0267) at (3,3);
    \coordinate (best0267) at (3,3);
    \coordinate (average0267) at (3,3.0);
    \coordinate (worst0267) at (3,3);

    \draw
    (best0267)++(-0.35,0) -- ++(0.7,0)  
    (worst0267)++(-0.35,0) -- ++(0.7,0) 
    (best0267) -- (worst0267);       
    \fill[black] (original0267) circle (2.5pt);
    \fill[white,draw=black] (average0267) circle (1.5pt);

    \coordinate (original0283) at (4,2);
    \coordinate (best0283) at (4,2);
    \coordinate (average0283) at (4,2.0);
    \coordinate (worst0283) at (4,2);

    \draw
    (best0283)++(-0.35,0) -- ++(0.7,0)  
    (worst0283)++(-0.35,0) -- ++(0.7,0) 
    (best0283) -- (worst0283);       
    \fill[black] (original0283) circle (2.5pt);
    \fill[white,draw=black] (average0283) circle (1.5pt);

    \coordinate (original0363) at (5,3);
    \coordinate (best0363) at (5,3);
    \coordinate (average0363) at (5,3.3);
    \coordinate (worst0363) at (5,4);

    \draw
    (best0363)++(-0.35,0) -- ++(0.7,0)  
    (worst0363)++(-0.35,0) -- ++(0.7,0) 
    (best0363) -- (worst0363);       
    \fill[black] (original0363) circle (2.5pt);
    \fill[white,draw=black] (average0363) circle (1.5pt);

    \coordinate (original0854) at (6,4);
    \coordinate (best0854) at (6,4);
    \coordinate (average0854) at (6,4.4);
    \coordinate (worst0854) at (6,5);

    \draw
    (best0854)++(-0.35,0) -- ++(0.7,0)  
    (worst0854)++(-0.35,0) -- ++(0.7,0) 
    (best0854) -- (worst0854);       
    \fill[black] (original0854) circle (2.5pt);
    \fill[white,draw=black] (average0854) circle (1.5pt);

    \coordinate (original0040) at (7,5);
    \coordinate (best0040) at (7,5);
    \coordinate (average0040) at (7,5.6);
    \coordinate (worst0040) at (7,6);

    \draw
    (best0040)++(-0.35,0) -- ++(0.7,0)  
    (worst0040)++(-0.35,0) -- ++(0.7,0) 
    (best0040) -- (worst0040);       
    \fill[black] (original0040) circle (2.5pt);
    \fill[white,draw=black] (average0040) circle (1.5pt);

    \coordinate (original0413) at (8,6);
    \coordinate (best0413) at (8,6);
    \coordinate (average0413) at (8,6.6);
    \coordinate (worst0413) at (8,7);

    \draw
    (best0413)++(-0.35,0) -- ++(0.7,0)  
    (worst0413)++(-0.35,0) -- ++(0.7,0) 
    (best0413) -- (worst0413);       
    \fill[black] (original0413) circle (2.5pt);
    \fill[white,draw=black] (average0413) circle (1.5pt);

    \coordinate (original0609) at (9,7);
    \coordinate (best0609) at (9,6);
    \coordinate (average0609) at (9,6.7);
    \coordinate (worst0609) at (9,7);

    \draw
    (best0609)++(-0.35,0) -- ++(0.7,0)  
    (worst0609)++(-0.35,0) -- ++(0.7,0) 
    (best0609) -- (worst0609);       
    \fill[black] (original0609) circle (2.5pt);
    \fill[white,draw=black] (average0609) circle (1.5pt);

    \coordinate (original0629) at (10,4);
    \coordinate (best0629) at (10,3);
    \coordinate (average0629) at (10,3.0);
    \coordinate (worst0629) at (10,3);

    \draw
    (best0629)++(-0.35,0) -- ++(0.7,0)  
    (worst0629)++(-0.35,0) -- ++(0.7,0) 
    (best0629) -- (worst0629);       
    \fill[black] (original0629) circle (2.5pt);
    \fill[white,draw=black] (average0629) circle (1.5pt);

    \coordinate (original0868) at (11,5);
    \coordinate (best0868) at (11,5);
    \coordinate (average0868) at (11,5.0);
    \coordinate (worst0868) at (11,5);

    \draw
    (best0868)++(-0.35,0) -- ++(0.7,0)  
    (worst0868)++(-0.35,0) -- ++(0.7,0) 
    (best0868) -- (worst0868);       
    \fill[black] (original0868) circle (2.5pt);
    \fill[white,draw=black] (average0868) circle (1.5pt);

    \coordinate (original0916) at (12,3);
    \coordinate (best0916) at (12,3);
    \coordinate (average0916) at (12,3.0);
    \coordinate (worst0916) at (12,3);

    \draw
    (best0916)++(-0.35,0) -- ++(0.7,0)  
    (worst0916)++(-0.35,0) -- ++(0.7,0) 
    (best0916) -- (worst0916);       
    \fill[black] (original0916) circle (2.5pt);
    \fill[white,draw=black] (average0916) circle (1.5pt);

    \coordinate (original0057) at (13,6);
    \coordinate (best0057) at (13,6);
    \coordinate (average0057) at (13,6.4);
    \coordinate (worst0057) at (13,7);

    \draw
    (best0057)++(-0.35,0) -- ++(0.7,0)  
    (worst0057)++(-0.35,0) -- ++(0.7,0) 
    (best0057) -- (worst0057);       
    \fill[black] (original0057) circle (2.5pt);
    \fill[white,draw=black] (average0057) circle (1.5pt);

    \coordinate (original0271) at (14,3);
    \coordinate (best0271) at (14,3);
    \coordinate (average0271) at (14,3.1);
    \coordinate (worst0271) at (14,4);

    \draw
    (best0271)++(-0.35,0) -- ++(0.7,0)  
    (worst0271)++(-0.35,0) -- ++(0.7,0) 
    (best0271) -- (worst0271);       
    \fill[black] (original0271) circle (2.5pt);
    \fill[white,draw=black] (average0271) circle (1.5pt);

    \coordinate (original0272) at (15,3);
    \coordinate (best0272) at (15,3);
    \coordinate (average0272) at (15,3.5);
    \coordinate (worst0272) at (15,4);

    \draw
    (best0272)++(-0.35,0) -- ++(0.7,0)  
    (worst0272)++(-0.35,0) -- ++(0.7,0) 
    (best0272) -- (worst0272);       
    \fill[black] (original0272) circle (2.5pt);
    \fill[white,draw=black] (average0272) circle (1.5pt);

    \coordinate (original0282) at (16,3);
    \coordinate (best0282) at (16,3);
    \coordinate (average0282) at (16,3.0);
    \coordinate (worst0282) at (16,3);

    \draw
    (best0282)++(-0.35,0) -- ++(0.7,0)  
    (worst0282)++(-0.35,0) -- ++(0.7,0) 
    (best0282) -- (worst0282);       
    \fill[black] (original0282) circle (2.5pt);
    \fill[white,draw=black] (average0282) circle (1.5pt);

    \coordinate (original0405) at (17,6);
    \coordinate (best0405) at (17,6);
    \coordinate (average0405) at (17,6.4);
    \coordinate (worst0405) at (17,7);

    \draw
    (best0405)++(-0.35,0) -- ++(0.7,0)  
    (worst0405)++(-0.35,0) -- ++(0.7,0) 
    (best0405) -- (worst0405);       
    \fill[black] (original0405) circle (2.5pt);
    \fill[white,draw=black] (average0405) circle (1.5pt);

    \coordinate (original0487) at (18,4);
    \coordinate (best0487) at (18,4);
    \coordinate (average0487) at (18,4.0);
    \coordinate (worst0487) at (18,4);

    \draw
    (best0487)++(-0.35,0) -- ++(0.7,0)  
    (worst0487)++(-0.35,0) -- ++(0.7,0) 
    (best0487) -- (worst0487);       
    \fill[black] (original0487) circle (2.5pt);
    \fill[white,draw=black] (average0487) circle (1.5pt);

    \coordinate (original0870) at (19,9);
    \coordinate (best0870) at (19,9);
    \coordinate (average0870) at (19,9.7);
    \coordinate (worst0870) at (19,11);

    \draw
    (best0870)++(-0.35,0) -- ++(0.7,0)  
    (worst0870)++(-0.35,0) -- ++(0.7,0) 
    (best0870) -- (worst0870);       
    \fill[black] (original0870) circle (2.5pt);
    \fill[white,draw=black] (average0870) circle (1.5pt);

    \coordinate (original0361) at (20,8);
    \coordinate (best0361) at (20,8);
    \coordinate (average0361) at (20,8.5);
    \coordinate (worst0361) at (20,9);

    \draw
    (best0361)++(-0.35,0) -- ++(0.7,0)  
    (worst0361)++(-0.35,0) -- ++(0.7,0) 
    (best0361) -- (worst0361);       
    \fill[black] (original0361) circle (2.5pt);
    \fill[white,draw=black] (average0361) circle (1.5pt);

    \coordinate (original0692) at (21,7);
    \coordinate (best0692) at (21,7);
    \coordinate (average0692) at (21,8.8);
    \coordinate (worst0692) at (21,9);

    \draw
    (best0692)++(-0.35,0) -- ++(0.7,0)  
    (worst0692)++(-0.35,0) -- ++(0.7,0) 
    (best0692) -- (worst0692);       
    \fill[black] (original0692) circle (2.5pt);
    \fill[white,draw=black] (average0692) circle (1.5pt);

    \coordinate (original0871) at (22,11);
    \coordinate (best0871) at (22,11);
    \coordinate (average0871) at (22,11.3);
    \coordinate (worst0871) at (22,12);

    \draw
    (best0871)++(-0.35,0) -- ++(0.7,0)  
    (worst0871)++(-0.35,0) -- ++(0.7,0) 
    (best0871) -- (worst0871);       
    \fill[black] (original0871) circle (2.5pt);
    \fill[white,draw=black] (average0871) circle (1.5pt);

    \coordinate (original0357) at (23,7);
    \coordinate (best0357) at (23,8);
    \coordinate (average0357) at (23,8.0);
    \coordinate (worst0357) at (23,8);

    \draw
    (best0357)++(-0.35,0) -- ++(0.7,0)  
    (worst0357)++(-0.35,0) -- ++(0.7,0) 
    (best0357) -- (worst0357);       
    \fill[black] (original0357) circle (2.5pt);
    \fill[white,draw=black] (average0357) circle (1.5pt);

    \coordinate (original0359) at (24,12);
    \coordinate (best0359) at (24,12);
    \coordinate (average0359) at (24,13.5);
    \coordinate (worst0359) at (24,17);

    \draw
    (best0359)++(-0.35,0) -- ++(0.7,0)  
    (worst0359)++(-0.35,0) -- ++(0.7,0) 
    (best0359) -- (worst0359);       
    \fill[black] (original0359) circle (2.5pt);
    \fill[white,draw=black] (average0359) circle (1.5pt);

    \coordinate (original0360) at (25,10);
    \coordinate (best0360) at (25,10);
    \coordinate (average0360) at (25,11.2);
    \coordinate (worst0360) at (25,13);

    \draw
    (best0360)++(-0.35,0) -- ++(0.7,0)  
    (worst0360)++(-0.35,0) -- ++(0.7,0) 
    (best0360) -- (worst0360);       
    \fill[black] (original0360) circle (2.5pt);
    \fill[white,draw=black] (average0360) circle (1.5pt);

    \coordinate (original0755) at (26,7);
    \coordinate (best0755) at (26,7);
    \coordinate (average0755) at (26,7.0);
    \coordinate (worst0755) at (26,7);

    \draw
    (best0755)++(-0.35,0) -- ++(0.7,0)  
    (worst0755)++(-0.35,0) -- ++(0.7,0) 
    (best0755) -- (worst0755);       
    \fill[black] (original0755) circle (2.5pt);
    \fill[white,draw=black] (average0755) circle (1.5pt);

    \coordinate (original0987) at (27,10);
    \coordinate (best0987) at (27,10);
    \coordinate (average0987) at (27,10.0);
    \coordinate (worst0987) at (27,10);

    \draw
    (best0987)++(-0.35,0) -- ++(0.7,0)  
    (worst0987)++(-0.35,0) -- ++(0.7,0) 
    (best0987) -- (worst0987);       
    \fill[black] (original0987) circle (2.5pt);
    \fill[white,draw=black] (average0987) circle (1.5pt);

    \coordinate (original0035) at (28,12);
    \coordinate (best0035) at (28,12);
    \coordinate (average0035) at (28,12.8);
    \coordinate (worst0035) at (28,14);

    \draw
    (best0035)++(-0.35,0) -- ++(0.7,0)  
    (worst0035)++(-0.35,0) -- ++(0.7,0) 
    (best0035) -- (worst0035);       
    \fill[black] (original0035) circle (2.5pt);
    \fill[white,draw=black] (average0035) circle (1.5pt);

    \coordinate (original0080) at (29,8);
    \coordinate (best0080) at (29,7);
    \coordinate (average0080) at (29,7.0);
    \coordinate (worst0080) at (29,7);

    \draw
    (best0080)++(-0.35,0) -- ++(0.7,0)  
    (worst0080)++(-0.35,0) -- ++(0.7,0) 
    (best0080) -- (worst0080);       
    \fill[black] (original0080) circle (2.5pt);
    \fill[white,draw=black] (average0080) circle (1.5pt);

    \coordinate (original0082) at (30,8);
    \coordinate (best0082) at (30,7);
    \coordinate (average0082) at (30,7.0);
    \coordinate (worst0082) at (30,7);

    \draw
    (best0082)++(-0.35,0) -- ++(0.7,0)  
    (worst0082)++(-0.35,0) -- ++(0.7,0) 
    (best0082) -- (worst0082);       
    \fill[black] (original0082) circle (2.5pt);
    \fill[white,draw=black] (average0082) circle (1.5pt);

    \coordinate (original1004) at (31,14);
    \coordinate (best1004) at (31,10);
    \coordinate (average1004) at (31,10.0);
    \coordinate (worst1004) at (31,10);

    \draw
    (best1004)++(-0.35,0) -- ++(0.7,0)  
    (worst1004)++(-0.35,0) -- ++(0.7,0) 
    (best1004) -- (worst1004);       
    \fill[black] (original1004) circle (2.5pt);
    \fill[white,draw=black] (average1004) circle (1.5pt);

    \coordinate (original0026) at (32,11);
    \coordinate (best0026) at (32,12);
    \coordinate (average0026) at (32,12.4);
    \coordinate (worst0026) at (32,13);

    \draw
    (best0026)++(-0.35,0) -- ++(0.7,0)  
    (worst0026)++(-0.35,0) -- ++(0.7,0) 
    (best0026) -- (worst0026);       
    \fill[black] (original0026) circle (2.5pt);
    \fill[white,draw=black] (average0026) circle (1.5pt);

    \coordinate (original0052) at (33,7);
    \coordinate (best0052) at (33,7);
    \coordinate (average0052) at (33,7.3);
    \coordinate (worst0052) at (33,8);

    \draw
    (best0052)++(-0.35,0) -- ++(0.7,0)  
    (worst0052)++(-0.35,0) -- ++(0.7,0) 
    (best0052) -- (worst0052);       
    \fill[black] (original0052) circle (2.5pt);
    \fill[white,draw=black] (average0052) circle (1.5pt);

    \coordinate (original0257) at (34,11);
    \coordinate (best0257) at (34,11);
    \coordinate (average0257) at (34,12.6);
    \coordinate (worst0257) at (34,16);

    \draw
    (best0257)++(-0.35,0) -- ++(0.7,0)  
    (worst0257)++(-0.35,0) -- ++(0.7,0) 
    (best0257) -- (worst0257);       
    \fill[black] (original0257) circle (2.5pt);
    \fill[white,draw=black] (average0257) circle (1.5pt);

    \coordinate (original0647) at (35,8);
    \coordinate (best0647) at (35,8);
    \coordinate (average0647) at (35,8.6);
    \coordinate (worst0647) at (35,9);

    \draw
    (best0647)++(-0.35,0) -- ++(0.7,0)  
    (worst0647)++(-0.35,0) -- ++(0.7,0) 
    (best0647) -- (worst0647);       
    \fill[black] (original0647) circle (2.5pt);
    \fill[white,draw=black] (average0647) circle (1.5pt);

    \coordinate (original0637) at (36,20);
    \coordinate (best0637) at (36,20);
    \coordinate (average0637) at (36,23.1);
    \coordinate (worst0637) at (36,29);

    \draw
    (best0637)++(-0.35,0) -- ++(0.7,0)  
    (worst0637)++(-0.35,0) -- ++(0.7,0) 
    (best0637) -- (worst0637);       
    \fill[black] (original0637) circle (2.5pt);
    \fill[white,draw=black] (average0637) circle (1.5pt);

    \coordinate (original0002) at (37,20);
    \coordinate (best0002) at (37,20);
    \coordinate (average0002) at (37,24.2);
    \coordinate (worst0002) at (37,35);

    \draw
    (best0002)++(-0.35,0) -- ++(0.7,0)  
    (worst0002)++(-0.35,0) -- ++(0.7,0) 
    (best0002) -- (worst0002);       
    \fill[black] (original0002) circle (2.5pt);
    \fill[white,draw=black] (average0002) circle (1.5pt);

    \coordinate (original0447) at (38,23);
    \coordinate (best0447) at (38,23);
    \coordinate (average0447) at (38,24.9);
    \coordinate (worst0447) at (38,26);

    \draw
    (best0447)++(-0.35,0) -- ++(0.7,0)  
    (worst0447)++(-0.35,0) -- ++(0.7,0) 
    (best0447) -- (worst0447);       
    \fill[black] (original0447) circle (2.5pt);
    \fill[white,draw=black] (average0447) circle (1.5pt);

    \coordinate (original0050) at (39,9);
    \coordinate (best0050) at (39,10);
    \coordinate (average0050) at (39,11.5);
    \coordinate (worst0050) at (39,12);

    \draw
    (best0050)++(-0.35,0) -- ++(0.7,0)  
    (worst0050)++(-0.35,0) -- ++(0.7,0) 
    (best0050) -- (worst0050);       
    \fill[black] (original0050) circle (2.5pt);
    \fill[white,draw=black] (average0050) circle (1.5pt);

    \coordinate (original0364) at (40,13);
    \coordinate (best0364) at (40,15);
    \coordinate (average0364) at (40,15.9);
    \coordinate (worst0364) at (40,17);

    \draw
    (best0364)++(-0.35,0) -- ++(0.7,0)  
    (worst0364)++(-0.35,0) -- ++(0.7,0) 
    (best0364) -- (worst0364);       
    \fill[black] (original0364) circle (2.5pt);
    \fill[white,draw=black] (average0364) circle (1.5pt);

    \coordinate (original0500) at (41,19);
    \coordinate (best0500) at (41,19);
    \coordinate (average0500) at (41,23.3);
    \coordinate (worst0500) at (41,27);

    \draw
    (best0500)++(-0.35,0) -- ++(0.7,0)  
    (worst0500)++(-0.35,0) -- ++(0.7,0) 
    (best0500) -- (worst0500);       
    \fill[black] (original0500) circle (2.5pt);
    \fill[white,draw=black] (average0500) circle (1.5pt);

    \coordinate (original0028) at (42,17);
    \coordinate (best0028) at (42,19);
    \coordinate (average0028) at (42,19.7);
    \coordinate (worst0028) at (42,21);

    \draw
    (best0028)++(-0.35,0) -- ++(0.7,0)  
    (worst0028)++(-0.35,0) -- ++(0.7,0) 
    (best0028) -- (worst0028);       
    \fill[black] (original0028) circle (2.5pt);
    \fill[white,draw=black] (average0028) circle (1.5pt);

    \coordinate (original0038) at (43,15);
    \coordinate (best0038) at (43,17);
    \coordinate (average0038) at (43,17.2);
    \coordinate (worst0038) at (43,18);

    \draw
    (best0038)++(-0.35,0) -- ++(0.7,0)  
    (worst0038)++(-0.35,0) -- ++(0.7,0) 
    (best0038) -- (worst0038);       
    \fill[black] (original0038) circle (2.5pt);
    \fill[white,draw=black] (average0038) circle (1.5pt);

    \coordinate (original0085) at (44,26);
    \coordinate (best0085) at (44,29);
    \coordinate (average0085) at (44,32.9);
    \coordinate (worst0085) at (44,38);

    \draw
    (best0085)++(-0.35,0) -- ++(0.7,0)  
    (worst0085)++(-0.35,0) -- ++(0.7,0) 
    (best0085) -- (worst0085);       
    \fill[black] (original0085) circle (2.5pt);
    \fill[white,draw=black] (average0085) circle (1.5pt);

    \coordinate (original0086) at (45,35);
    \coordinate (best0086) at (45,40);
    \coordinate (average0086) at (45,44.9);
    \coordinate (worst0086) at (45,49);

    \draw
    (best0086)++(-0.35,0) -- ++(0.7,0)  
    (worst0086)++(-0.35,0) -- ++(0.7,0) 
    (best0086) -- (worst0086);       
    \fill[black] (original0086) circle (2.5pt);
    \fill[white,draw=black] (average0086) circle (1.5pt);

    \coordinate (original0259) at (46,16);
    \coordinate (best0259) at (46,16);
    \coordinate (average0259) at (46,16.3);
    \coordinate (worst0259) at (46,17);

    \draw
    (best0259)++(-0.35,0) -- ++(0.7,0)  
    (worst0259)++(-0.35,0) -- ++(0.7,0) 
    (best0259) -- (worst0259);       
    \fill[black] (original0259) circle (2.5pt);
    \fill[white,draw=black] (average0259) circle (1.5pt);

    \coordinate (original0260) at (47,16);
    \coordinate (best0260) at (47,16);
    \coordinate (average0260) at (47,16.5);
    \coordinate (worst0260) at (47,18);

    \draw
    (best0260)++(-0.35,0) -- ++(0.7,0)  
    (worst0260)++(-0.35,0) -- ++(0.7,0) 
    (best0260) -- (worst0260);       
    \fill[black] (original0260) circle (2.5pt);
    \fill[white,draw=black] (average0260) circle (1.5pt);

    \coordinate (original0261) at (48,16);
    \coordinate (best0261) at (48,16);
    \coordinate (average0261) at (48,16.5);
    \coordinate (worst0261) at (48,18);

    \draw
    (best0261)++(-0.35,0) -- ++(0.7,0)  
    (worst0261)++(-0.35,0) -- ++(0.7,0) 
    (best0261) -- (worst0261);       
    \fill[black] (original0261) circle (2.5pt);
    \fill[white,draw=black] (average0261) circle (1.5pt);

    \coordinate (original0030) at (49,19);
    \coordinate (best0030) at (49,21);
    \coordinate (average0030) at (49,23.2);
    \coordinate (worst0030) at (49,25);

    \draw
    (best0030)++(-0.35,0) -- ++(0.7,0)  
    (worst0030)++(-0.35,0) -- ++(0.7,0) 
    (best0030) -- (worst0030);       
    \fill[black] (original0030) circle (2.5pt);
    \fill[white,draw=black] (average0030) circle (1.5pt);

    \coordinate (original0315) at (50,31);
    \coordinate (best0315) at (50,34);
    \coordinate (average0315) at (50,38.7);
    \coordinate (worst0315) at (50,43);

    \draw
    (best0315)++(-0.35,0) -- ++(0.7,0)  
    (worst0315)++(-0.35,0) -- ++(0.7,0) 
    (best0315) -- (worst0315);       
    \fill[black] (original0315) circle (2.5pt);
    \fill[white,draw=black] (average0315) circle (1.5pt);

    \coordinate (original0638) at (51,44);
    \coordinate (best0638) at (51,44);
    \coordinate (average0638) at (51,62.0);
    \coordinate (worst0638) at (51,74);

    \draw
    (best0638)++(-0.35,0) -- ++(0.7,0)  
    (worst0638)++(-0.35,0) -- ++(0.7,0) 
    (best0638) -- (worst0638);       
    \fill[black] (original0638) circle (2.5pt);
    \fill[white,draw=black] (average0638) circle (1.5pt);

    \coordinate (original0011) at (52,20);
    \coordinate (best0011) at (52,23);
    \coordinate (average0011) at (52,25.1);
    \coordinate (worst0011) at (52,27);

    \draw
    (best0011)++(-0.35,0) -- ++(0.7,0)  
    (worst0011)++(-0.35,0) -- ++(0.7,0) 
    (best0011) -- (worst0011);       
    \fill[black] (original0011) circle (2.5pt);
    \fill[white,draw=black] (average0011) circle (1.5pt);

    \coordinate (original0200) at (53,36);
    \coordinate (best0200) at (53,40);
    \coordinate (average0200) at (53,43.4);
    \coordinate (worst0200) at (53,47);

    \draw
    (best0200)++(-0.35,0) -- ++(0.7,0)  
    (worst0200)++(-0.35,0) -- ++(0.7,0) 
    (best0200) -- (worst0200);       
    \fill[black] (original0200) circle (2.5pt);
    \fill[white,draw=black] (average0200) circle (1.5pt);

    \coordinate (original0475) at (54,22);
    \coordinate (best0475) at (54,23);
    \coordinate (average0475) at (54,24.8);
    \coordinate (worst0475) at (54,26);

    \draw
    (best0475)++(-0.35,0) -- ++(0.7,0)  
    (worst0475)++(-0.35,0) -- ++(0.7,0) 
    (best0475) -- (worst0475);       
    \fill[black] (original0475) circle (2.5pt);
    \fill[white,draw=black] (average0475) circle (1.5pt);

    \coordinate (original0430) at (55,32);
    \coordinate (best0430) at (55,38);
    \coordinate (average0430) at (55,41.4);
    \coordinate (worst0430) at (55,43);

    \draw
    (best0430)++(-0.35,0) -- ++(0.7,0)  
    (worst0430)++(-0.35,0) -- ++(0.7,0) 
    (best0430) -- (worst0430);       
    \fill[black] (original0430) circle (2.5pt);
    \fill[white,draw=black] (average0430) circle (1.5pt);

    \coordinate (original0431) at (56,32);
    \coordinate (best0431) at (56,36);
    \coordinate (average0431) at (56,40.1);
    \coordinate (worst0431) at (56,45);

    \draw
    (best0431)++(-0.35,0) -- ++(0.7,0)  
    (worst0431)++(-0.35,0) -- ++(0.7,0) 
    (best0431) -- (worst0431);       
    \fill[black] (original0431) circle (2.5pt);
    \fill[white,draw=black] (average0431) circle (1.5pt);

    \coordinate (original0581) at (57,27);
    \coordinate (best0581) at (57,17);
    \coordinate (average0581) at (57,17.5);
    \coordinate (worst0581) at (57,18);

    \draw
    (best0581)++(-0.35,0) -- ++(0.7,0)  
    (worst0581)++(-0.35,0) -- ++(0.7,0) 
    (best0581) -- (worst0581);       
    \fill[black] (original0581) circle (2.5pt);
    \fill[white,draw=black] (average0581) circle (1.5pt);

    \coordinate (original0365) at (58,40);
    \coordinate (best0365) at (58,50);
    \coordinate (average0365) at (58,57.1);
    \coordinate (worst0365) at (58,63);

    \draw
    (best0365)++(-0.35,0) -- ++(0.7,0)  
    (worst0365)++(-0.35,0) -- ++(0.7,0) 
    (best0365) -- (worst0365);       
    \fill[black] (original0365) circle (2.5pt);
    \fill[white,draw=black] (average0365) circle (1.5pt);

    \coordinate (original0599) at (59,39);
    \coordinate (best0599) at (59,44);
    \coordinate (average0599) at (59,48.9);
    \coordinate (worst0599) at (59,56);

    \draw
    (best0599)++(-0.35,0) -- ++(0.7,0)  
    (worst0599)++(-0.35,0) -- ++(0.7,0) 
    (best0599) -- (worst0599);       
    \fill[black] (original0599) circle (2.5pt);
    \fill[white,draw=black] (average0599) circle (1.5pt);

    \coordinate (original0747) at (60,33);
    \coordinate (best0747) at (60,34);
    \coordinate (average0747) at (60,37.3);
    \coordinate (worst0747) at (60,40);

    \draw
    (best0747)++(-0.35,0) -- ++(0.7,0)  
    (worst0747)++(-0.35,0) -- ++(0.7,0) 
    (best0747) -- (worst0747);       
    \fill[black] (original0747) circle (2.5pt);
    \fill[white,draw=black] (average0747) circle (1.5pt);

    \coordinate (original0335) at (61,37);
    \coordinate (best0335) at (61,36);
    \coordinate (average0335) at (61,38.7);
    \coordinate (worst0335) at (61,41);

    \draw
    (best0335)++(-0.35,0) -- ++(0.7,0)  
    (worst0335)++(-0.35,0) -- ++(0.7,0) 
    (best0335) -- (worst0335);       
    \fill[black] (original0335) circle (2.5pt);
    \fill[white,draw=black] (average0335) circle (1.5pt);

    \coordinate (original0362) at (62,39);
    \coordinate (best0362) at (62,39);
    \coordinate (average0362) at (62,41.2);
    \coordinate (worst0362) at (62,43);

    \draw
    (best0362)++(-0.35,0) -- ++(0.7,0)  
    (worst0362)++(-0.35,0) -- ++(0.7,0) 
    (best0362) -- (worst0362);       
    \fill[black] (original0362) circle (2.5pt);
    \fill[white,draw=black] (average0362) circle (1.5pt);

    \coordinate (original0951) at (63,39);
    \coordinate (best0951) at (63,38);
    \coordinate (average0951) at (63,40.4);
    \coordinate (worst0951) at (63,43);

    \draw
    (best0951)++(-0.35,0) -- ++(0.7,0)  
    (worst0951)++(-0.35,0) -- ++(0.7,0) 
    (best0951) -- (worst0951);       
    \fill[black] (original0951) circle (2.5pt);
    \fill[white,draw=black] (average0951) circle (1.5pt);

    \coordinate (original0501) at (64,48);
    \coordinate (best0501) at (64,55);
    \coordinate (average0501) at (64,64.4);
    \coordinate (worst0501) at (64,78);

    \draw
    (best0501)++(-0.35,0) -- ++(0.7,0)  
    (worst0501)++(-0.35,0) -- ++(0.7,0) 
    (best0501) -- (worst0501);       
    \fill[black] (original0501) circle (2.5pt);
    \fill[white,draw=black] (average0501) circle (1.5pt);

    \coordinate (original0492) at (65,52);
    \coordinate (best0492) at (65,76);
    \coordinate (average0492) at (65,86.1);
    \coordinate (worst0492) at (65,98);

    \draw
    (best0492)++(-0.35,0) -- ++(0.7,0)  
    (worst0492)++(-0.35,0) -- ++(0.7,0) 
    (best0492) -- (worst0492);       
    \fill[black] (original0492) circle (2.5pt);
    \fill[white,draw=black] (average0492) circle (1.5pt);

    \coordinate (original0491) at (66,57);
    \coordinate (best0491) at (66,72);
    \coordinate (average0491) at (66,84.7);
    \coordinate (worst0491) at (66,99);

    \draw
    (best0491)++(-0.35,0) -- ++(0.7,0)  
    (worst0491)++(-0.35,0) -- ++(0.7,0) 
    (best0491) -- (worst0491);       
    \fill[black] (original0491) circle (2.5pt);
    \fill[white,draw=black] (average0491) circle (1.5pt);

    \coordinate (original0014) at (67,59.7058823529);
    \coordinate (best0014) at (67,89.7058823529);
    \coordinate (average0014) at (67,64.9705882353);
    \coordinate (worst0014) at (67,100);

    \draw
    (best0014)++(-0.35,0) -- ++(0.7,0)  
    (worst0014)++(-0.35,0) -- ++(0.7,0) 
    (best0014) -- (worst0014);       
    \node at (original0014) {$\ast$};
    \fill[white,draw=black] (average0014) circle (1.5pt);

    \coordinate (original0559) at (68,77.0700636943);
    \coordinate (best0559) at (68,91.7197452229);
    \coordinate (average0559) at (68,96.1146496815);
    \coordinate (worst0559) at (68,100);

    \draw
    (best0559)++(-0.35,0) -- ++(0.7,0)  
    (worst0559)++(-0.35,0) -- ++(0.7,0) 
    (best0559) -- (worst0559);       
    \node at (original0559) {$\ast$};
    \fill[white,draw=black] (average0559) circle (1.5pt);
\end{tikzpicture}
